\documentclass{aims}
\usepackage{amsmath}
  \usepackage{paralist}
  \usepackage{graphics} 
  \usepackage{epsfig} 
 \usepackage[colorlinks=true]{hyperref}
\hypersetup{urlcolor=blue, citecolor=red}

\def\currentvolume{X}
 \def\currentissue{X}
  \def\currentyear{200X}
   \def\currentmonth{XX}
    \def\ppages{X--XX}

\newtheorem{theorem}{Theorem}[section]
\newtheorem{corollary}{Corollary}

\theoremstyle{definition}
\newtheorem{definition}[theorem]{Definition}
\newtheorem{remark}{Remark}

\title[Max Regularity and Analytic Semigroups]
      {Continuous Maximal Regularity and Analytic Semigroups}

\author[Jeremy LeCrone and Gieri Simonett]{}


 \email{jeremy.lecrone@vanderbilt.edu}
 \email{gieri.simonett@vanderbilt.edu}

\begin{document}
\maketitle

\centerline{\scshape Jeremy LeCrone }
\medskip
{\footnotesize
 \centerline{Department of Mathematics}
   \centerline{Vanderbilt University}
   \centerline{Nashville, TN 37240, USA}
} 

\medskip

\centerline{\scshape Gieri Simonett}
\medskip
{\footnotesize
 \centerline{Department of Mathematics}
   \centerline{Vanderbilt University}
   \centerline{Nashville, TN 37240, USA}
}



\begin{abstract}
In this paper we establish a result regarding the connection between continuous maximal regularity
and generation of analytic semigroups on a pair of densely embedded Banach spaces. More precisely,
we show that continuous maximal regularity for a closed operator $A: E_1 \rightarrow E_0$
implies that $A$ generates a strongly continuous analytic semigroup on $E_0$ with domain equal $E_1$.
\end{abstract}


\section{Introduction}

In this paper we consider the following abstract inhomogeneous equation
\begin{equation}\label{Problem}
\begin{cases}
\frac{d}{dt}u(t) - Au(t) = f(t), \quad\text{$t \in J$},\\
u(0)=u_0 .
\end{cases}
\end{equation}
Here we are assuming that $J:=[0,T]$ for a fixed $T>0$, or $J:=\mathbb{R}_+=[0,\infty)$, and that $A: D(A) \subset E_0 \rightarrow E_0$ is a closed, densely defined operator on the complex Banach space $(E_0, \| \cdot \|_0)$. It then follows that
$(E_1, \| \cdot \|_1) := (D(A), \| \cdot \|_{D(A)})$, equipped with the graph norm $\| \cdot \|_{D(A)}$, is also a complex Banach space and $E_1$ is densely (continuously) embedded in $E_0$. We assume that $f : J \rightarrow E_0$ and $u_0\in E_0$ are known, so that the unknown quantity in \eqref{Problem} is the function $u : J \rightarrow E_0$. We say that $u$ is a \emph{classical solution} to \eqref{Problem} on $J$ if $u \in C(J,E_0) \cap C^1( \dot{J}, E_0) \cap C( \dot{J}, E_1) $ and $u$ satisfies \eqref{Problem} for $t \in \dot{J} := J \setminus \{ 0 \}$.

In approaching this problem, it would be beneficial to know from the outset that the operator $A$ generates a strongly continuous semigroup $\{ e^{tA}: t \geq 0 \}$ on $E_0$. Then one can derive an explicit representation for solutions to \eqref{Problem}. More precisely, if a solution $u$ of \eqref{Problem} exists then it is represented by the variation of parameters formula as
\begin{equation}
\label{Eqn:VarOfParams}
u(t)= e^{tA}u_0 + \int_0^t {e^{(t - \tau)A} f(\tau) \, d \tau} \qquad \text{for} \quad t \in J,
\end{equation}
under the assumption that $A$ generates a C$_0$-semigroup.
However, as we will show in our main result, one can do away with this assumption if the operator is known to satisfy the conditions of continuous maximal regularity.

In case $A$ generates an analytic semigroup, sufficient conditions for
continuous maximal regularity were first obtained by Da Prato and Grisvard \cite{DPG79}.
The results contained in this paper show that the assumption of $A$ generating an
analytic semigroup is also necessary for the Da Prato-Grisvard result.

The theory of maximal regularity has proved itself to be an indispensable tool
in the analysis of nonlinear parabolic problems, and it has been used by many authors.
In the case of continuous maximal regularity we just mention the references
\cite{DaLu88, ES97, ES98, EMS98, Si95}.


\section{Maximal Regularity}\label{sec:MaximalRegularity}
In the remainder of this paper we will use the following notation:
if $E_0$ and $E_1$ are Banach spaces we say that
$(E_0,E_1)$ is a pair of embedded Banach spaces if $E_1\subset E_0$
and  the natural injection $i:E_1\to E_0$ is continuous.
If, in addition, $E_1$ is a dense subset of $E_0$ then
we call $(E_0,E_1)$ a pair of densely embedded spaces.
Suppose $(E_0,E_1)$ is a pair of embedded Banach spaces
and $A:E_1\to E_0$ is a linear operator. Then we can also interpret
$A:E_1\subset E_0\to E_0$ as an unbounded linear operator on $E_0$ with domain $D(A)=E_1$.
It is then meaningful to say that $A$ is a closed operator on $E_0$.
For a given pair $(E_0,E_1)$ of embedded Banach spaces
we shall often consider linear operators
$A:E_1\to E_0$ with the property that $A\in\mathcal{B}(E_1,E_0)$
(i.e. $A:E_1\to E_0$ is a bounded operator)
and the additional
property that $A$ is closed on $E_0$.
These combined properties turn out to be equivalent to the
fact that $E_1$ and $(D(A),\|\cdot\|_{D(A)})$ carry equivalent norms,
see \cite[Lemma I.1.1.2]{Amann95}.

Next we give a general definition of maximal regularity, and then move on
to focus on the more specific continuous maximal regularity.

\begin{definition}
Let $J := [0,T]$ for $T > 0$, or $J := \mathbb{R}_+ = [0,\infty)$, and assume that
$(E_0,E_1)$ is a pair of densely embedded Banach spaces.
Moreover, suppose that
$(\mathbb{E}_0(J),\mathbb{E}_1(J))$
is a pair of Banach spaces such that
\begin{equation*}\label{eqn:BlockESpaces}
\mathbb{E}_0(J) \subset L_{1,\text{loc}}(J,E_0)  \qquad
\mathbb{E}_1(J) \subset W^{1}_{1,\text{loc}}(J,E_0) \cap L_{1,\text{loc}}(J,E_1).
\end{equation*}
Define the \emph{trace operator} $\gamma : \mathbb{E}_1(J) \rightarrow E_0$ by $\gamma u := u(0)$ for $u \in \mathbb{E}_1(J)$
and the \emph{trace space} $\gamma \mathbb{E}_1$ by
\begin{equation*}
\begin{split}
\gamma \mathbb{E}_1 &:= \{ x \in E_0 : x = \gamma u \text{ for some } u \in \mathbb{E}_1(J) \},\\
\| x \|_{\gamma} &:= \| x \|_{\gamma \mathbb{E}_1} := \inf \{ \| u \|_{\mathbb{E}_1(J)} : u \in \mathbb{E}_1(J) \text{ and } \gamma u = x \}.
\end{split}
\end{equation*}
For $A\in\mathcal{B}(E_1,E_0)$
we say that $(\mathbb{E}_0(J),\mathbb{E}_1(J))$ is a \emph{pair of maximal regularity for $A$} if
\begin{equation*}\label{eqn:MaximalRegularity}
\left( \frac{d}{dt} - A, \gamma \right) \in \text{Isom} \big( \mathbb{E}_1(J), \mathbb{E}_0(J) \times \gamma \mathbb{E}_1 \big).
\end{equation*}
That is, $(\mathbb{E}_0(J),\mathbb{E}_1(J))$ enjoys the property of maximal regularity if, for every $(f,x) \in \mathbb{E}_0(J) \times \gamma \mathbb{E}_1$, there exists a unique function $u \in \mathbb{E}_1(J)$ such that $u$ satisfies
\begin{equation}
\label{eqn:MaxRegEqn}
\begin{cases}
\frac{d}{dt}u(t) - A u(t) = f(t) & \text{for $t \in J$}\\
\gamma u = u(0) = x
\end{cases}
\end{equation}
and the mapping $[(f,x) \mapsto u]$ is continuous.
For a pair $(\mathbb{E}_0(J),\mathbb{E}_1(J))$ of maximal regularity for $A$, we define the \emph{solution operator}
\begin{equation*}
K_A: \mathbb{E}_0(J) \rightarrow \mathbb{E}_1(J) \quad \text{by}
\quad K_Af := \left( \frac{d}{dt} - A, \gamma \right)^{-1} (f,0).
\end{equation*}
\end{definition}
We say that the pair enjoys the property of \emph{continuous} maximal regularity if
\begin{equation*}
(\mathbb{E}_0(J), \, \mathbb{E}_1(J)) = (BU\!C(J,E_0), BU\!C^1(J,E_0) \cap BU\!C(J,E_1)).
\end{equation*}
Here  $BU\!C(J,E)$ denotes the space of all bounded uniformly continuous functions
$u:J\to E$, and $BU\!C^1(J, E)$ stands for all functions in $BU\!C(J,E)$ whose derivative also
shares this property. In case $J=[0,T]$ is a compact interval we will use the shorter notation
$C$ and $C^1$ in place of $BU\!C$ and $BU\!C^1$, respectively.
We refer to \cite{Amann95} for more background information on maximal regularity.

Notice that $(\mathbb{E}_0(J), \mathbb{E}_1(J))$ a pair of continuous maximal regularity for $A$ implies existence of a unique classical solution $u$ to \eqref{Problem} for every $f \in \mathbb{E}_0(J)$ and $x \in E_1$, since $E_1$ and $\gamma \mathbb{E}_1$ coincide in this case, and there exists a positive constant M, independent of $f$ and $x$, such that the estimate
\begin{equation}
\label{MaxReg}
\sup_{t \in J} \big( \| \dot{u}(t) \|_{0} + \| u(t) \|_{1} \big) \leq M \left( \sup_{t \in J} \| f(t) \|_{0} + \| x \|_{1} \right)
\end{equation}
holds. Furthermore, it follows from Theorem~\ref{MainResult} below and \eqref{Eqn:VarOfParams} that the solution operator produces the function
\begin{equation*}
(K_Af)(t) = \int_0^t {e^{(t - \tau)A} f(\tau) \, d \tau} \qquad \text{for} \quad t \in J.
\end{equation*}
\noindent
The following theorem is a modification of a result by Dore~\cite{Dore93} and  Pr\"uss~\cite{Pruss02}, who demonstrated a proof in the case of $L_p$ maximal regularity.
%


\bigskip
\begin{theorem}
\label{MainResult}
Fix $T > 0$ (or $T = \infty$) and $J := [0,T]$ (or $J := \mathbb{R}_+)$ and let $(E_0, E_1)$
be a pair of densely embedded Banach spaces. If $A \in \mathcal{B}(E_1, E_0)$ is
a closed operator on $E_0$ and
\[
\big( \mathbb{E}_0(J),\mathbb{E}_1(J) \big) = \left( BU \! C(J, E_0), BU \! C^1(J, E_0) \cap BU \! C(J, E_1) \right)
\]
is a pair of continuous maximal regularity for $A$, then $A$ generates a strongly continuous
analytic semigroup on $E_0$. Moreover, in the case $J = \mathbb{R}_+$, it follows that $s(A) < 0$ where
$s(A) := \sup \{ {\rm Re}\, \mu: \mu \in \sigma (A) \}$ is the spectral bound of $A$.
\end{theorem}


\begin{proof}
We begin by demonstrating the result for $T < \infty$ and then we extend to the unbounded interval $\mathbb{R}_+$.
In particular, we show that there exist constants $\omega \in \mathbb{R}$ and $N \geq 1$ so that
\begin{equation*}
\{ \mu \in \mathbb{C}: \text{Re} \, \mu >  \omega \} \subset \rho(A) \quad \text{and}
\quad \|(\mu - A)^{-1} \|_{\mathcal{B}(E_0)} \leq \frac{N}{1 + | \mu |} \quad \text{for} \quad \text{Re} \, \mu > \omega \, .
\end{equation*}
\\
\emph{\bfseries Claim 1:} There exists $\omega_1 \geq 0$ sufficiently large so that the  a priori estimate
\begin{equation}
\label{a-priori}
\|x\|_1+|\mu| \|x\|_0\le C\|(\mu-A)x\|_0,\quad x\in E_1,\quad \text{Re} \, \mu > \omega_1,
\end{equation}
holds. In particular, $(\mu - A) : E_1 \rightarrow E_0$ is injective for $\text{Re} \, \mu > \omega_1 \, .$
\bigskip
\begin{itemize}
\item[]
Let $x \in E_1, \, \mu \in \mathbb{C}$ and set $v_{\mu}(t) := v_{\mu}(t,x) = e^{\mu t} x$.
Then $v_{\mu} \in {\mathbb E}_1(J)$ and satisfies
\begin{equation*}
\dot{v}_{\mu}(t) - A v_{\mu}(t) = g_{\mu}(t), \quad v_{\mu}(0) = x,
\end{equation*}
where $g_{\mu}(t) = e^{\mu t}(\mu - A) x \in \mathbb{E}_0(J)$.
By assumption, we see that inequality (\ref{MaxReg}) holds and implies
\begin{equation}
\label{I1}
\sup_{t \in J} e^{\text{Re}\,\mu t} \big(\|x\|_1 + |\mu|\|x\|_0 \big)
\leq M \left( \sup_{t \in J} e^{\text{Re}\,\mu t} \| (\mu - A)x \|_0 + \| x \|_1 \right).
\end{equation}
Now, choosing $\omega_1 \geq 0$ large enough so that $\displaystyle 2M \leq \sup_{t \in J} e^{\omega_1 t}$ we obtain
\begin{equation}
\label{I2}
\| x \|_1 + | \mu | \| x \|_0 \leq 2M \| (\mu - A)x \|_0 , \quad \text{Re} \, \mu > \omega_1,
\end{equation}
thus establishing the claim.
\end{itemize}
\bigskip
\noindent
Now let $x \in E_0$, $\mu \in \mathbb{C}_+ := \{ z \in \mathbb{C}: \text{Re} \, \mu > 0 \}$ and define the functions
\begin{equation*}
f_{\mu}(t) := e^{-\mu t}x \in BU \!C(\mathbb{R}_+,E_0), \qquad u_{\mu} := u_{\mu}(\cdot,x) := K_A f_{\mu} \in \mathbb{E}_1(J) \, .
\end{equation*}
Notice that continuity of the embedding $E_1 \hookrightarrow E_0$ and continuity of $K_A$ imply existence of positive constants $c_1$ and $c_2$ so that
	\begin{equation*}
	\label{embedding}
	\sup_{t \in J} \| u_{\mu}(t) \|_0 \leq c_1 \sup_{t \in J} \| u_{\mu} (t) \|_1 \leq c_1 \| u_{\mu} \|_{\mathbb{E}_1} \leq c_2 \| f_{\mu} \|_{\mathbb{E}_0} = c_2 \sup_{t \in J} \| f_{\mu}(t) \|_0 \leq c_2 \| x \|_0 \, .
	\end{equation*}
In particular, we see that
\begin{equation}
\label{uMuBound}
\| u_{\mu} (t) \|_0 \leq c_2 \| x \|_0 \qquad \mu \in \mathbb{C}_+, \; t \in J.
\end{equation}
Also, for $\mu \in \mathbb{C}_+$ we note that $\bar{\mu} \in \mathbb{C}_+$ and define the functions
$U_{\mu}: E_0 \rightarrow E_1$ and $V_{\mu}: E_0 \rightarrow E_0$ as
\begin{equation}
\label{eqn:BigUmu}
\begin{aligned}
  U_{\mu}x :&= 2 {\rm Re}\, \mu \int_0^T e^{- \mu t} u_{\bar{\mu}}(t,x) \, dt,\\
   V_{\mu}x :&= \frac{2  {\rm Re} \, \mu \, e^{- \mu T}}{ \left(1 - e^{- 2 {\rm Re} \, \mu T} \right)} \, u_{\bar{\mu}} (T,x) \, .
\end{aligned}
\end{equation}\\
\emph{\bfseries Claim 2:}
There exists $\omega_2 \geq 0$ sufficiently large so that
$(\mu - A) : E_1 \rightarrow E_0$ is surjective for $\text{Re} \, \mu > \omega_2 \, .$
\begin{itemize}

	\item[] 	
	First notice that
	\begin{equation*}
	\| V_{\mu} x \|_{0} = \frac{2  {\rm Re} \, \mu \, e^{- {\rm Re} \, \mu T}}
	{ \left(1 - e^{- 2 {\rm Re} \, \mu T} \right)} \, \| u_{\bar{\mu}} (T,x) \|_0
	\leq c_2\frac{2   {\rm Re} \, \mu \, e^{- {\rm Re} \, \mu T}}{ \left(1 - e^{- 2 {\rm Re} \, \mu T} \right)} \, \| x \|_0 \, ,
	\end{equation*}
	converges to 0 as Re$\, \mu \rightarrow \infty$.
	So we fix $\omega_2 \geq 0$ large enough that $\| V_{\mu} \|_{\mathcal{B}(E_0)}
	< \frac{1}{2}$ for $\text{Re} \, \mu > \omega_2 \, .$
  Multiplying the relation
  \[
   \frac{d}{dt} u_{\bar{\mu}}(t,x) - A u_{\bar{\mu}}(t,x) = e^{- \bar{\mu} t} x
  \]
	by $2 ({\rm Re} \, \mu) e^{- \mu t}$ and integrating over the interval $[0,T]$ yields
	\begin{equation*}
	(\mu - A)U_{\mu} x = \big(1 - e^{- 2 {\rm Re} \, \mu T} \big)
	\, \left( I - V_{\mu} \right) \, x, \qquad x \in E_0, \, \mu \in \mathbb{C}_+ \, .
	\end{equation*}
	Here we use the fact that $A$ is a closed operator and so we are free to interchange $A$ with integration.
	Therefore, choosing $\text{Re} \, \mu > \omega_2$ we conclude that
	$(\mu - A)U_{\mu}$ is invertible by the Neumann series.
	 However, invertibility of $(\mu - A)U_{\mu}$ implies surjectivity of $(\mu - A)$ and so the claim is established.
	\end{itemize}
\bigskip
\noindent
With claims 1 and 2 established, we see that choosing $\omega \geq \omega_1 \vee \omega_2$ implies $\{ \mu \in \mathbb{C}: {\rm Re} \mu > \omega \} \subset \rho(A)$ and the inequality
\[
\| ( \mu - A )^{-1} \|_{\mathcal{B}(E_0)} \leq \frac{2M(1 \vee c_1)}{1 + | \mu |}, \qquad {\rm Re}\, \mu > \omega \, ,
\]
follows from (\ref{I2}). It is well known that these properties are sufficient for $A$ to generate an analytic semigroup on $E_0$, c.f. \cite[Theorem 1.2.2]{Amann95}, which concludes the proof for the finite interval $J = [0,T]$.
\medskip\\
Now we consider the case of continuous maximal regularity on $\mathbb{R}_+$. More precisely,
assume that $\big( \mathbb{E}_0 (\mathbb{R}_+), \mathbb{E}_1 (\mathbb{R}_+) \big) =
\big( BU \!C( \mathbb{R}_+, E_0), BU \!C^1( \mathbb{R}_+, E_0) \cap BU \!C(\mathbb{R}_+, E_1) \big)$
is a pair of maximal regularity for $A$.
\medskip\\
\noindent
\emph{\bfseries Claim 3:}
For any $T > 0, \, \big( \mathbb{E}_0 (J), \mathbb{E}_1 (J) \big)$ is a pair of
maximal regularity for $A$, where $J = [0,T]$.
\begin{itemize}
\item[]
	Note that any function $f \in C(J, E_0)$ can be
	extended to $Ef \in BU \!C( \mathbb{R}_+, E_0)$ by setting $(Ef)(t) = f(T)$ for $t \geq T$.
	Moreover, denoting by $R$ the restriction operator to the interval $J$,
	we see that
	$u := R \big( \frac{d}{dt} - A, \gamma \big)^{-1} (Ef, x) \in \mathbb{E}_1(J) \,$
	is a solution to problem (\ref{eqn:MaxRegEqn}) on $J$,
	for $f \in C(J,E_0)$ and $x \in E_1$.
	It remains to show that this solution $u$ is unique in $\mathbb{E}_1(J)$.
	In fact, it suffices to show that $u \equiv 0$ is the unique solution
	in $\mathbb{E}_1(J)$ to problem (\ref{eqn:MaxRegEqn}) with $(f,x) = (0,0)$.
	By way of contradiction, suppose $\tilde{u} \in \mathbb{E}_1(J)$
	solves (\ref{eqn:MaxRegEqn}) with $(f,x) = (0,0)$ and there exists $t_1 \in (0,T)$ such
	that $\tilde{u}(t_1) \not= 0$. Then, define $v := \big( \frac{d}{dt} - A, \gamma \big)^{-1} (0, \tilde{u}(t_1))$,
	so that $v \in \mathbb{E}_1(\mathbb{R}_+)$ and $v$ satisfies $\dot{v} - Av = 0$ on $\mathbb{R}_+$
	and  $v(0) = \tilde{u}(t_1)$. Define
	\[
	w(t) :=
	\begin{cases}
		\tilde u(t), & \text{$t \in [0, t_1],$}\\
		v(t - t_1), & \text{$ t \in [t_1, \infty),$}
	\end{cases}
	\]
	and it follows that $w \in \mathbb{E}_1 (\mathbb{R}_+)$ and $w$ satisfies
	(\ref{eqn:MaxRegEqn}) with $(f,x)=(0,0)$, on $\mathbb{R}_+$.
	However, note that $w(t_1) = \tilde u(t_1) \not= 0$, so that
	$w \not\equiv 0$, a contradiction.
\end{itemize}
\bigskip
\noindent
Thus, by Claims 1 and 2 above, we see that
$A$ generates an analytic semigroup on $E_0$.
It remains to consider the spectral bound $s(A)$.

\medskip
\noindent
\emph{\bfseries Claim 4:}
The spectral bound of $A$ is strictly negative, i.e. $s(A) < 0$.
\begin{itemize}
\item[]
	By assumption of continuous maximal regularity on $\mathbb{R}_+$, and by Claim 3,
	there exists a positive constant $M$ (independent of $T > 0$) such that
	\begin{equation*}
	\label{T-indepenent}
	\| u \|_{{\mathbb E}_1(J)} \leq M \big( \| f \|_{{\mathbb E}_0(J)} + \| x \|_1 \big), \quad
	(f, x) \in {\mathbb E}_0(J) \times E_1,
	\end{equation*}
	for the solution $u = R \big( \frac{d}{dt} - A, \gamma \big)^{-1} (Ef, x)$ of (\ref{eqn:MaxRegEqn}) on $J = [0,T]$.
	From this estimate, we conclude that inequality (\ref{I1}) holds for any $T > 0$ and
	so Claim 1 is true for $\omega_1 = 0$ in this case.
	Furthermore, setting
	\begin{equation*}
	U_{\mu}x := 2 {\rm Re}\, \mu \int_0^\infty e^{- \mu t} u_{\bar{\mu}}(t,x) \, dt
	\qquad \text{and} \qquad	V_\mu :=0,
	\end{equation*}
	we see that Claim 2 holds for $\omega_2 = 0$.	Therefore, it follows that
	\begin{equation*}
	{\mathbb C}_+\subset\rho(A)\quad\text{and}\quad \| (\mu-A)^{-1} \|_{\mathcal{B}(E_0)}
 	\leq \frac{c}{1 + | \mu |},\quad {\rm Re}\,\mu >0\,.
	\end{equation*}
	Since this estimate is uniform in ${\rm Re}\,\mu >0$, we conclude that
	the imaginary axis must be contained in the resolvent set of $A$
	and the estimate continues to hold for ${\rm Re}\,\mu\ge 0$.
	This implies that $\sigma(A)\subset[{\rm Re}\,\mu<0]$ and $s(A)<0$ as claimed.
\end{itemize}
\vspace{-5mm}
\end{proof}

\begin{remark}
{\bf (a)}
If the underlying Banach spaces $E_0$ and $E_1$ are real, we can apply
Theorem \ref{MainResult} to the complexification of $E_0$, $E_1$ and $A$. Then,
by restriction back to the real case we derive results in the original, real-valued, setting.
Hence, we lose no generality here by focusing on only the complex case.\\
\goodbreak
\noindent
{\bf (b)}
Suppose  $(C(J,E_0), C^1(J,E_0) \cap C(J,E_1))$ is a pair of maximal regularity
for some 
$A \in \mathcal{B}(E_1, E_0)$.
Then by a result of Baillon \cite{Ba80} (see also \cite{EG92})
either $E_{1}=E_{0}$,
or $E_{0}$ contains a closed subspace which is isomorphic to $c_{0}$, the space of null sequences.
In particular, if $E_{0}$ is reflexive or weakly sequentially compact then
continuous maximal regularity cannot occur.\\
\goodbreak
\noindent
{\bf (c)}
If one already knows that $A$ generates a $C_0$-semigroup on $E_0$,
then DaPrato and Grisvard proved that continuous maximal regularity
implies that $A$ is the generator of an analytic semigroup,
see also \cite[Proposition III.3.1.1]{Amann95}.
%
%
%
%
\\
\goodbreak
\noindent
{\bf (d)}
For the reader's convenience we briefly describe a situation which shows that many
interesting operators $A$ give rise to continuous maximal regularity.
Assume that $A\in\mathcal{H}(E_{1},E_{0})$ (i.e.
$-A$ generates a strongly continuous analytic semigroup on $E_0$ with $D(A) = E_1$) and define
\begin{equation*}
\begin{split}
&E_{2}:=E_{2}(A):=( D(A^{2}),\|\cdot\, \|_{E_{2}}), \\
&\|\cdot \|_{E_{2}}:= \|\cdot \|_{E_{2}(A)}:=\|A\cdot \|_{E_{1}} +
\|\cdot \|_{E_{1}}.
\end{split}
\end{equation*}
Then $(E_{2},\|\cdot\|_{E_{2}})$ is a Banach space, and $E_2$ is densely embedded in $E_1$.
We set
\begin{equation*}
\begin{split}
&E_{\theta}:=(E_{0},E_{1})_{\theta}, \\
&E_{1+\theta}:=E_{1+\theta}(A):=(E_{1},E_{2}(A))_{\theta},\qquad
\theta \in (0,1), \\ &A_{\theta}:=\text{the maximal
$E_{\theta}$-realization of $A$},
\end{split}
\end{equation*}
where $(\cdot,\cdot)_\theta$ denotes the continuous interpolation method
of Da Prato and Grisvard \cite{DPG79},
see also \cite{Amann95, Ang90, CleSim01, Lun95}.
It is then well-known that  $A_{\theta}\in
\mathcal{H}(E_{\theta},E_{1+\theta})$, and it turns out that
$A_{\theta}$ gives rise to continuous maximal regularity.
\medskip\\
{\bf Theorem.} {\rm (Da Prato-Grisvard)}
Suppose that $\theta\in (0,1)$, $T > 0$. Let $J = [0,T]$. Then
$ \displaystyle
(\mathbb{E}_0(J),\mathbb{E}_1(J)):=
(C(J,E_{\theta}),
C^{1}(J,E_{\theta})\cap C(J,E_{1+\theta}))
$
is a pair of maximal regularity for $A_{\theta}$. 
\end{remark}
\medskip
\begin{remark}
{\bf (a) Maximal Regularity With Prescribed Singularity}

\noindent
We demonstrate that Theorem~\ref{MainResult} continues to hold in a more general setting,
namely a space of continuous functions which allow for singularities at zero.
These spaces have for instance been studied in \cite{Amann95, Ang90, CleSim01}.
This setting is well adapted for the study of quasilinear parabolic equations,
see for instance \cite {CleSim01, EMS98}.

\bigskip
\noindent
Assume that $\sigma \in (0,1)$, $E$ is a Banach space over
$\mathbb{K} \,\big( = \mathbb{R} \; \text{or} \; \mathbb{C} \big)$, $J=[0,T]$
for some $T>0$, and we set $\dot J:=J\setminus\{0\}$.
Then define
\begin{equation*}
\begin{aligned}
BU\!C_{1- \sigma}(J,E):=\!\big\{u\in C(\dot J,E):
[t\mapsto t^{1- \sigma}u(t)]\in BU\!C(\dot J,E),
\lim_{t\to 0^{+}}t^{1- \sigma}\|u(t)\|=0\big\},
\end{aligned}
\end{equation*}
which is a Banach space when equipped with the norm
\begin{equation*}
\|u\|_{C_{1-\sigma}}:=\sup_{t\in \dot J}t^{1-\sigma}\|u(t)\|_{E},
\quad u\in BU\!C_{1-\sigma}(J,E), \quad \sigma\in (0,1).
\end{equation*}
Further we introduce the subspace
\begin{equation*}
BU\!C^{1}_{1-\sigma}(J,E):=\{u\in C^{1}(\dot J,E)
: u,\dot{u} \in BU\!C_{1-\sigma}(J,E)\}.
\end{equation*}
Now, given a pair of densely embedded Banach spaces $(E_0,E_1)$
and $\sigma \in (0,1)$, we consider the pair of associated function spaces
\begin{equation}
\label{2.1}
\begin{aligned}
\mathbb{E}_{0}(J):&=BU\!C_{1-\sigma}(J,E_{0}),\\
\mathbb{E}_{1}(J):&=BU\!C^{1}_{1-\sigma}(J,E_{0})\cap BU\!C_{1-\sigma}(J,E_{1}),
\end{aligned}
\end{equation}
where $\mathbb{E}_{1}(J)$ is a Banach space with the norm
\begin{equation*}
\|u\|_{\mathbb{E}_{1}(J)}:=
\sup_{t\in\dot J}t^{1-\sigma}(\| \dot{u}(t)\|_{E_{0}}+\|u(t)\|_{E_{1}}).
\end{equation*}
It is well-known that the trace space is then given by
$\gamma{\mathbb E}_1=(E_0,E_1)_\sigma$, $\sigma\in (0,1),$
where $(\cdot,\cdot)_\sigma$ denotes the continuous interpolation method,
see \cite[Section 1.2.2] {Lun95} for instance.


\begin{corollary}
Let $(E_0, E_1)$ be a pair of densely embedded Banach spaces, $J = [0,T]$ with $T > 0$,
and $A\in\mathcal{B}(E_1,E_0)$ a closed operator on $E_0$. If 
\[
\big( \mathbb{E}_0(J),\mathbb{E}_1(J) \big) = \left( BU \! C_{1 - \sigma}(J, E_0), BU \! C_{1 - \sigma}^1(J, E_0) \cap BU \! C_{1 - \sigma}(J, E_1) \right)
\]
is a pair of maximal regularity for $A$, for some $\sigma \in (0,1]$, then $A$ generates a strongly continuous
analytic semigroup on $E_0$.
\end{corollary}

\begin{proof}
The methods presented in the proof of Theorem~\ref{MainResult}
apply in this setting with minor modifications. We provide an outline of the
proof in this case in order to illuminate the necessary modifications.
For Claim 1, fix $x \in E_1$
and define $v_{\mu} \in \mathbb{E}_1(J)$,
$g_{\mu}\in \mathbb{E}_0(J)$ as before.
Using the continuous embedding $E_1\hookrightarrow \gamma{\mathbb E}_1$ we derive as above the inequality
\begin{equation*}
\sup_{t \in J} t^{1 - \sigma} e^{{\rm Re} \, \mu t} \big( |\mu| \| x \|_0 + \| x \|_1 \big)
\le M  \big( \sup_{t \in J} t^{1 - \sigma}e^{{\rm Re} \, \mu t} \| (\mu - A) x \|_0 + \| x \|_1 \big),
\end{equation*}
which is the analog of (\ref{I1}) in this setting.
Choosing $\omega_1 $ so that
$ 2M \leq \sup_{t \in J} t^{1 - \sigma} e^{\omega_1 t},$ we see that
(\ref{I2}) holds as before.

Meanwhile, for Claim 2 we consider $f_{\mu}, \, u_{\mu}, \, U_{\mu}$
and $V_{\mu}$ defined as above, for $\mu \in \mathbb{C}_+$.
Recalling the continuity and embedding constants $c_1, \, c_2$, we see, for $t \in \dot{J}$,
\begin{equation*}
t^{1 - \sigma}\| u_{\bar{\mu}}(t) \|_0
\le  \| u_{\bar{\mu}} \|_{\mathbb{E}_0(J)}
\le  c_1 \| u_{\bar{\mu}} \|_{\mathbb{E}_1(J)}
\le c_2 \| f_{\bar{\mu}} \|_{\mathbb{E}_0(J)}
\le c_2 T^{1-\sigma}\| x \|_0.
\end{equation*}
This shows that (\ref{uMuBound}) holds for $t = T$ and the remainder of the proof
follows exactly as in Theorem~\ref{MainResult}. \qedhere

\end{proof}

\vspace{2mm}
\noindent
{\bfseries (b) $L_p$ Maximal Regularity}

\noindent
For a given pair of densely embedded Banach spaces $(E_0,E_1)$ one defines
the pair of functions spaces
\begin{equation*}
\begin{split}
{\mathbb E}_1(J):&=W^1_p(J,E_0)\cap L_p(J,E_1), \\
{\mathbb E}_0(J):&=L_p(J,E_0).
\end{split}
\end{equation*}
It is well-known that the trace space is then given by
$\gamma{\mathbb E}_1=(E_0,E_1)_{1-1/p,p}$.

An inspection of the proof of Theorem~\ref{MainResult} shows that the same
methods apply to this $L_p$-maximal regularity setting, with minor modifications.
Moreover, the methods presented herein
considerably simplify the arguments in \cite{Dore93,Pruss02}.
\end{remark}


\section*{Acknowledgments}
We thank the anonymous reviewer for valuable suggestions.
We would also like to thank Jan Pr\"uss and Mathias Wilke for helpful conversations.
\bigskip


\medskip
Received July 2010; revised March 2011.
\medskip

\end{document}